\documentclass[11pt,a4paper,reqno]{amsart}
\usepackage[english]{babel}
\usepackage[applemac]{inputenc}
\usepackage[T1]{fontenc}
\usepackage{bbm}
\usepackage{cite}
\usepackage{float, verbatim}
\usepackage{amsmath}
\usepackage{amssymb}
\usepackage{amsthm}
\usepackage{amsfonts}
\usepackage{graphicx}
\usepackage{mathtools}
\usepackage{enumitem}
\usepackage[colorlinks = true, citecolor = black]{hyperref}
\usepackage{color}
\usepackage{tikz}
\usepackage{xcolor}
\usepackage{pgfplots}
\usepackage{caption}
\usepackage{subcaption}
\usepackage{enumerate}
\usepackage{url}

\newcommand{\ubox}{\overline{\dim_{\mathrm{B}}}}

\newcommand{\lbox}{\underline{\dim_{\mathrm{B}}}}
\newcommand{\boxd}{\dim_{\mathrm{B}}}

\newcommand{\Haus}{\dim_{\mathrm{H}}}

\newtheorem*{thm*}{Theorem}

\newtheorem{thm}{Theorem}[section]

\newtheorem{defn}[thm]{Definition}

\newtheorem{prop}[thm]{Proposition}
\newtheorem{conj}[thm]{Conjecture}
\newtheorem{rem}[thm]{Remark}
\newtheorem{ques}[thm]{Question}

\begin{document}
	
	\title{Additive properties of numbers with restricted digits}
	
	\author{Han Yu}
	\address{Han Yu\\
		Department of Pure Mathematics and Mathematical Statistics\\University of Cambridge\\CB3 0WB \\ UK }
	\curraddr{}
	\email{hy351@maths.cam.ac.uk}
	\thanks{ }
	
	
	
	\date{}
	
	\dedicatory{}
	
	\begin{abstract}
		In this paper, we consider some additive properties of integers with restricted digit expansions. Let $b\geq 3$ be an integer and $B_b$ be the set of integers whose base $b$ expansions have only digits $\{0,1\}.$ Let $a,b,c$ be three integers greater than $2.$ We give some estimates on the size of $(B_{a}+B_{b})\cap B_{c}.$ In particular, under mild conditions, $(B_{a}+B_{b})\cap B_{c}$ is a very thin set in the following sense that for each $\epsilon>0,$ as $N\to\infty,$
		\[
		\#((B_{a}+B_{b})\cap B_{c} \cap [1,N])=O(N^{\epsilon}).
		\]
	\end{abstract}
	
	\maketitle
	\allowdisplaybreaks
	\section{Introduction}
	In this paper, we discuss a problem of Furstenberg type and an application in number theory. Before jumping into the world of ergodicity, dynamical systems and fractals, we first mention a number theoretic problem which may be interesting on its own. We start with a definition.
	\begin{defn}
		Let $b\geq 3$ be an integer. Let $B_b$ be the set of positive integers whose base $b$ expansion contain only digits $\{0,1\}.$
	\end{defn}
	We consider the following problem.
	\begin{ques}
		Find integers $a,b,c$ and $(x,y,z)\in B_{a}\times B_{b}\times B_{c}$ such that
		\[
		x+y=z.
		\]
	\end{ques}
	We remark that if we fix $a,b,c$ to be small integers then it is possible that there are infinitely many solutions to $x+y=z, (x,y,z)\in B_{a}\times B_{b}\times B_{c}.$ For example, we have the following result \cite[Theorem 2.4]{Y21}
	\begin{thm*}
		There are infinitely many elements $(x,y,z)\in B_{3}\times B_{4}\times B_{5}$ with $x+y=z.$
	\end{thm*}
	However, if
	\begin{align*}
	\frac{\log 2}{\log a}+\frac{\log 2}{\log b}+\frac{\log 2}{\log c}<1\tag{*}
	\end{align*}
	then we suspect that there are only finitely many such solutions. Let us first make some simple observations. Clearly, for large integer $N$, the number of $(x,y,z)\in B_{a}\times B_{b}\times B_{c}\cap [0,N]^3$ is roughly
	\[
	N^{\log 2/\log a+\log 2/\log b+\log 2/\log c}.
	\] 
	We want to count the number of above points which are also contained in the plane $\{x+y=z\}.$ Now let us consider the family of planes $\{x+y=z+k\}$ for integers $k\in [-2N,2N].$ All the points in $B_{a}\times B_{b}\times B_{c}$ lies in precisely one of the planes. The averaged number of points on planes is of order
	\[
	N^{\log 2/\log a+\log 2/\log b+\log 2/\log c}/N.
	\]
	Thus if we unrealistically believe that ${x+y=z}$ behaves like the average, then the condition (*) in above tells us that there are not many points on
	\[
	B_{a}\times B_{b}\times B_{c}\cap [0,N]^3\cap \{x+y=z\}.
	\]
	We will make the above intuitive guess more precisely. Before that, we make the following definition.
	\begin{defn}
		We say that integers $b_1,\dots,b_k$ are pairwisely multiplicatively independent if  $\log b_i/\log b_j$ is irrational for each pair of different indices $i,j\in \{1,\dots,k\}.$
	\end{defn}
	
	\begin{conj}\label{SFermat}
		Let $a,b,c\geq 3$ be pairwisely multiplicatively independent integers such that
		\[
		\frac{\log 2}{\log a}+\frac{\log 2}{\log b}+\frac{\log 2}{\log c}<1.
		\]
		Then    there are at most finitely many integers $(x,y,z)\in B_{a}\times B_{b}\times B_{c}$ such that
		\[
		x+y=z.
		\]
	\end{conj}
	Towards this direction, we will prove the following result which to some extend, 'confirms' Conjecture \ref{SFermat} up to a very small uncertainty.
	\begin{thm}\label{MAIN}
		Let $a,b,c\geq 3$ be pairwisely multiplicatively independent integers such that
		\[
		\frac{\log 2}{\log a}+\frac{\log 2}{\log b}+\frac{\log 2}{\log c}<1.
		\]
		Then, for each $\epsilon>0$, the number of solutions $(x,y,z)\in B_a\times B_b\times B_c$, $x+y=z,$  with $x,y,z\geq 0$ and $z\leq N$ is $O(N^\epsilon).$
	\end{thm}
	\section{Proof of Theorem \ref{MAIN}: A Furstenberg's problem}
	In order to prove Theorem \ref{MAIN} we need some tools from fractal geometry and dynamical systems. A central result we need is as follows.
	\begin{thm}\label{Dyn}
		Let $a,b,c$ be three pairwisely multiplicatively independent integers. Let $A_a,A_b,A_c\subset [0,1]$ be closed $\times a,b,c \mod 1$ invariant set respectively. Suppose that
		\[
		\Haus A_a+\Haus A_b+\Haus A_c<1.
		\]
		For each $\Delta>0,$ we consider the set $S_\Delta$ of directions in $S^2$ whose all three coordinates have absolute values larger than $\Delta$ (bounded away from being parallel to coordinate planes). For each $\epsilon>0$ there is a constant $C$ such that for all planes $P$ with normal direction in $S_\Delta$ (we will call these planes to be $\delta$-generic) and all $r\in (0,1)$ we have the following box counting estimate for all $r>0,$
		\[
		N(P\cap A_a\times A_b\times A_c, r)\leq C r^{-\epsilon}.
		\]
	\end{thm}
	Here $N(X,r)$ is the box covering number of a set $X$ with cubes of side length $r$, see Section \ref{Pre}.  Theorem \ref{Dyn} is related to a higher dimensional version of the Furstenberg intersection problem. The strong form of Furstenberg intersection problem asks whether $A_a\cap A_b$ is finite under the condition that $\Haus A_a+\Haus A_b<1.$ For more details on Furstenberg intersection problem, see \cite{Fu1}, \cite{Fu2}, \cite{Sh}, \cite{Wu}, \cite{Y19}.
	
	At this stage, we mention that Theorem \ref{Dyn} is dealing with fibres of linear projections from $\mathbb{R}^3$ to $\mathbb{R}.$ For this reason, \cite[Theorem 1.11, Lemma 1.8]{Sh} can be used to prove this result. For general cases (See the next paragraph with $m=1,d\geq 3$. Those cases were considered in \cite{Y21} and \cite{BHY19}), methods in \cite{Sh} cannot be directly applied. Thus we will introduce an alternative approach by modifying the arguments in \cite[Section 10]{Y19}. As this result is essentially known, we will only outline a sketch of the proof and provide in detail all additional ingredients which were not provided in \cite{Y19}.
	
	Using our method, it seems to be quite likely that one can directly generalize the above theorem by considering the intersections between any affine $m$-subspace in $\mathbb{R}^d$ and $A_{a_1}\times A_{a_2}\times\dots\times A_{a_d},$ where $1\leq m<d$ is an integer and $a_1,\dots,a_d$ are  multiplicatively independent integers. Results in this direction can lead us to a generalization of Theorem \ref{MAIN} concerning linear forms of numbers with restricted digits. 
	
	Assuming Theorem \ref{Dyn}, the proof of Theorem \ref{MAIN} is straightforward.
	\begin{proof}[Proof of Theorem \ref{MAIN} based on Theorem \ref{Dyn}]
		We want to study the set $H\cap B_{a}\times B_{b}\times B_{c}$ where $H$ is the plane $\{x+y=z\}.$ Let $n$ be an integer and we consider
		\[
		H_n=\{(x,y,z)\in H\cap B_{a}\times B_{b}\times B_{c}: z\in [c^n,c^{n+1}-1]\}.
		\]
		Now we apply the map $T_n=(\times a^{-[n\log c/\log a]}, \times b^{-[n\log c/\log b]}, \times c^{-n})$ on $H_n.$ The image $T_n(H_n)$ is contained in the plane $T_n(H).$ The normal direction of this plane is the direction of 
		\[
		(a^{[n\log c/\log a]}, b^{[n\log c/\log b]}, -c^{n}).
		\]
		We can normalize the last component (which will not change its direction) and obtain
		\[
		(a^{-\{n\log c/\log a\}},b^{-\{n\log c/\log b\}},-1).
		\]
		This vector is contained in $[a^{-1},1]\times [b^{-1},1]\times \{1\}.$  Then we see that this normal direction in $S^2$ has coordinates which are all away from being zero, say, the absolute values are greater than $\Delta$ for a  constant $\Delta>0$. Let $(x,y,z)\in T_n(H_n)$ then we see that $x,y,z$ contain only digits $0,1$ in their base $a,b,c$ expansions respectively. The set of all numbers  whose base $a$ expansions contain only digits $0,1$ has Hausdorff dimension $\log 2/\log a.$ As 
		\[
		\frac{\log 2}{\log a}+\frac{\log 2}{\log b}+\frac{\log 2}{\log c}<1
		\]
		we can use Theorem \ref{Dyn}. First, we see that $T_n(H_n)$ is contained in 
		\[
		T_n(H)\cap T_n(B_{a}\times B_{b}\times (B_{c}\cap [c^n,c^{n+1}))).
		\]
		Let $A_a$ ($A_b,A_c$) be the set of positive numbers whose base $a$ ($b,c$ respectively) expansions only contain digits $0,1.$ We see that
		\[
		T_n(B_{a}\times B_{b}\times B_{c})\subset A_a\times A_b\times (A_c\cap [1,c)).
		\]
		Now, as $H$ is the plane $x+y=z,$ we see that $H_n$ is a bounded set. More precisely, we have
		\[
		H_n\subset [0,c^{n+1})\times [0,c^{n+1})\times [c^n,c^{n+1}).
		\]
		Thus we see that $T_n(H_n)$ is contained in
		\[
		[0,ca^{\{n\log c/\log a\}}]\times [0,cb^{\{n\log c/\log b\}}]\times [1,c]\subset [0,ac]\times [0,cb]\times [0,c].
		\]
		We can now apply Theorem \ref{Dyn} with the plane $T_n(H)$ and the set 
		\[
		A_a\times A_b\times A_c\cap [0,ac]\times [0,cb]\times [0,c],
		\]   
		which is contained in a union of  finitely (with an amount which is absolutely bounded) many translations of $(A_a\cap [0,1])\times (A_b\cap [0,1])\times (A_c\cap [0,1]).$	As a result, for $\epsilon>0,$ there is an integer $N$ such that whenever $n\geq N$ we have
		\[
		N(T_n(H_n), 2^{-n})\leq 2^{\epsilon n}.
		\]
		However, we see that $c^n T_n$ maps the unit cube to a rectangular shape whose sides are at least $1$ and at most $\max\{a,b\}.$ Therefore we see that there are constants $C,C'>0$ and
		\[
		\#H_n\leq Cc^{C'\epsilon n}.
		\]
		Here the constants $C,C'$ do not depend on $n,\epsilon.$ From here we see that as $n\to\infty$
		\[
		\sum_{1\leq k\leq n} \#H_n\leq c^{C''\epsilon n}
		\]
		for another constant $C''>0.$
		This concludes the result as we can choose $\epsilon$ to be arbitrarily small.
		%
	\end{proof}
	\section{Preliminaries}\label{Pre}
	From now on, we focus on proving Theorem \ref{Dyn}. As we mentioned before, the strategy will be similar to that in \cite[Section 10]{Y19} apart from a few additional materials. We will present those materials in this section.
	\subsection{$\times p\mod 1$ invariant sets}\label{INV}
	In this paper, given an integer $p\geq 2$, we use $A_p$ to denote an arbitrary closed $\times p\mod 1$ invariant subset of $[0,1]$. This is to say, for all $a\in A_p$, $\{pa\}\in A_p$, where $\{x\}$ is the fractional part of $x$. We say that $A_p$ is strictly invariant if $a\in A_p\iff \{pa\}\in A_p.$ For each closed $\times p\mod 1$ invariant set $A_p$, it is known (\cite[Theorem 5.1]{Fu}) that $\Haus A_p=\ubox A_p$, where $\dim$ with different subscripts are notions of dimensions which will be defined below.
	\subsection{Dimensions}\label{DIM}
	We briefly introduce some notions of dimensions.  For more details on the Hausdorff and box dimensions, see \cite[Chapters 2,3]{Fa} and \cite[Chapters 4,5]{Ma1}. For the Assouad dimension, see \cite{F}. We shall use $N(F,r)$ for the minimal covering number of a set $F$ in $\mathbb{R}^n$ with closed cubes of side length $r>0$. 
	
	\subsubsection{Hausdorff dimension}
	
	Let $g: [0,1)\to [0,\infty)$ be a continuous function such that $g(0)=0$. Then for all $\delta>0$ we define the following quantity
	\[
	\mathcal{H}^g_\delta(F)=\inf\left\{\sum_{i=1}^{\infty}g(\mathrm{diam} (U_i)): \bigcup_i U_i\supset F, \mathrm{diam}(U_i)<\delta\right\}.
	\]
	The $g$-Hausdorff measure of $F$ is
	\[
	\mathcal{H}^g(F)=\lim_{\delta\to 0} \mathcal{H}^g_{\delta}(F).
	\]
	When $g(x)=x^s$ then $\mathcal{H}^g=\mathcal{H}^s$ is the $s$-Hausdorff measure and Hausdorff dimension of $F$ is
	\[
	\Haus F=\inf\{s\geq 0:\mathcal{H}^s(F)=0\}=\sup\{s\geq 0: \mathcal{H}^s(F)=\infty          \}.
	\]
	\subsubsection{Box dimensions}
	The upper box dimension of a bounded set $F$ is
	\[
	\overline{\boxd} F=\limsup_{r\to 0} \left(-\frac{\log N(F,r)}{\log r}\right).
	\]
	Similarly, the lower box dimension of $F$ is
	\[
	\lbox F=\liminf_{r\to 0} \left(-\frac{\log N(F,r)}{\log r}\right).
	\]
	If the limsup and liminf are equal, we call this value the box dimension of $F$, and we denote it as $\boxd F.$
	
	\subsection{Sparse set} We also need the notion of sparseness which was introduced in \cite{Y19} for subsets of $\mathbb{R}.$ We now generalize this notion to $\mathbb{R}^d,d\geq 2.$
	\subsubsection{Densities of integer sequences}\label{DEN}
	We also work with various notions of densities of integer sequences.
	Let $W\subset\mathbb{N}$ be a sequence of integers, and we denote
	\[
	\#_nW=\#\{i\in [1,n]: i\in W\}.
	\]
	Now we recall two notions of density for integer sequences.
	\begin{defn}
		The upper natural density of $W$ is defined as
		\[
		\overline{d}(W)=\limsup_{n\to\infty} \frac{\#_nW}{n}.
		\]
		Similarly, we define the lower natural density by replacing the above $\limsup$ with $\liminf$ and write it as $\underline{d}(W)$. If these two numbers coincide we call it the natural density of $W$ and write it as $d(W).$
	\end{defn}
	
	\subsubsection{The big $O$ and small $o$ notations}
	Let $f,g: \mathbb{N}\to [0,\infty)$ be two functions. We write
	$
	f=O(g)
	$
	if there exists positive number $C>0$ such that
	$
	f(k)\leq Cg(k)
	$
	for all $k\in\mathbb{N}.$
	Similarly, we write
	$
	f=o(g)
	$
	if for any $\epsilon>0$ there exists $N\in\mathbb{N}$ such that for all $k\geq N$ we have
	$
	f(k)\leq \epsilon g(k).
	$
	In some occasions there is another parameter set $S$ and we have functions $f,g:\mathbb{N}\times S\to [0,\infty).$ For each $c\in S$ and we write $f=O_c(g), o_c(g)$ to indicate that the above tendencies depend on the choice of $c.$ We say that $f=O(g), o(g)$ uniformly for $c\in S$ if the above tendencies do not depend on the choice of $c.$
	\subsubsection{Sparseness and box counting estimates}
	\begin{defn}
		Let $X\subset\mathbb{R}^d$ be a compact set. Let $x\in \mathbb{R}^d.$ We say that $X$ is sparse around $x$ if the following sequence has natural density zero,
		\[
		W(X,x)=\{k\in\mathbb{N}:X\cap (B(x,2^{-k})\setminus B(x,2^{-k-1}))\neq \emptyset\}.
		\]
	\end{defn}
	If $x\notin X,$ then $W(X,x)$ is a finite sequence. Thus, $X$ is not sparse around $X$ only if $x\in X.$ The relation between sparseness and box counting numbers can be established via the following result whose proofs can be found in \cite{VK}, \cite[Theorem 6.10]{Lu} and \cite{KRS}.
	\begin{thm}\label{DOUB}
		Let $d\geq 1$ be an integer. Let $X\subset\mathbb{R}^d$ be a compact set. Then there is a doubling probability  measure supported on $X$. Namely, there is a measure $\mu\in\mathcal{P}(X)$ and there exists an absolute constant (called the doubling constant for $\mathbb{R}^d$) $D_d\geq 1$ such that for all $x\in X$ and $r>0$,
		\[
		0<\mu(B(x,2r))\leq D_d\mu(B(x,r))< \infty.
		\] 
	\end{thm}
	In what follows, we write $\mathbf{0}$ for the origin of $\mathbb{R}^d.$
	\begin{prop}\label{USPA}
		Let $X\subset B(\mathbf{0},0.5)$ be a closed sparse set. Assume further that the zero density of $W(X,x)$ holds uniformly for all $x\in X.$ That is, for each $\epsilon>0$ there is a constant $C$ such that for all integer $N\geq C$ and $x\in X$
		\[
		\# W(X,x)\cap [1,N]\leq \epsilon N.
		\]
		Then for each $\epsilon'>0$, there is a constant $C'$ such that for all $r\in (0,1)$
		\[
		N(X,r)\leq C' r^{-\epsilon'}.
		\]
		Here the choice of $C'$ does not depend on $X.$
	\end{prop}
	\begin{proof}
		We see that the following set has $0$ upper natural density uniformly across $x\in X$,
		\[
		W(X,x)=\{k\in\mathbb{N}:X\cap (B(x,2^{-k})\setminus B(x,2^{-k-1}))\neq \emptyset\}.
		\]
		To bound the box counting numbers of $X$ we shall use Theorem \ref{DOUB} and find a doubling (with doubling constant $D>0$ depending on $d$) probability measure $\mu$ supported on $X$.  Let $x\in X$ be arbitrarily chosen and for any integer $n\geq 0$ we can find a nested sequence of balls $x\in B(x,2^{-n})\subset\dots\subset B(x,1).$ Since we assumed that $X\subset B(\mathbf{0},0.5)$ therefore we see that $\mu(B(x,1))=1.$ Now we make use of the sparseness of $X$. Observe that $X\cap B(x,2^{-j})=X\cap B(x,2^{-j-1})$ if $j\notin W(X,x).$ Then we write
		\[
		\mu(B(x,2^{-n}))=\mu(B(x,1))\prod_{j=0}^{n-1} \frac{\mu(B(x,2^{-j-1}))}{\mu(B(x,2^{-j}))}.
		\]
		If $j\notin W(X,x)$ then $X\cap B(x,2^{-j})\setminus B(x,2^{-j-1})=\emptyset$ therefore we see that,
		\[
		\frac{\mu(B(x,2^{-j-1}))}{\mu(B(x,2^{-j}))}=1,
		\]
		otherwise if $j\in W(X,x)$ we can still write
		\[
		\frac{\mu(B(x,2^{-j-1}))}{\mu(B(x,2^{-j}))}\geq D^{-1}.
		\]
		Since $W(X,x)$ has natural density $0$ uniformly across $x\in X,$ we see that for all $\epsilon>0$ there exist a $N_\epsilon$ such that for all $x\in X, N\geq N_\epsilon$ we have
		\[
		\#W(X,x)\cap [1,N]\leq \epsilon N.
		\]
		Then we see that for all $N\geq N_\epsilon$
		\[
		\mu(B(x,2^{-N}))\geq D^{-\epsilon N}.
		\]
		By Besicovitch covering Theorem (\cite[Chapter 2, Section 7]{Ma1}), we can cover $X$ with balls of radius $2^{-N-1}$ with bounded overlapping multiplicity. That is for each $x\in X$; there are at most $M$ balls containing it. Here $M$ is a constant that depends only on $d$. Denote the collection of such balls as $\mathcal{N}_{N+1}$, then for any $B\in \mathcal{N}_{N+1}$ there is a point $x\in B\cap X$ such that $B\subset B(x,2^{-N})$ and therefore $\mu(B)\geq D^{-\epsilon N}.$ Since $\mu$ is a probability measure we see that
		\[
		\#\mathcal{N}_{N+1}\leq M D^{\epsilon N}.
		\]
		From here we see that there is a constant $C$ and for all integer $N\geq N_\epsilon$ we have $N(X,2^{-N-1})\leq C 2^{\epsilon N}$ This concludes the proof by transforming the constants properly.
	\end{proof}
	Now we introduce a notion of uniform sparseness.
	\begin{defn}\label{US}
		Let $\mathcal{E}=\{E\}_{i\in\mathcal{I}}$ be a collection of closed subsets of $\mathbb{R}^d.$ We say that $\mathcal{E}$ is uniformly sparse if for each $\epsilon>0$, there is an integer $N_\epsilon$ such that for each $N\geq N_\epsilon$ and all $i\in\mathcal{I}$ and $x\in E_i$ we have
		\[
		\#W(E_i,x)\cap [1,N]\leq \epsilon N.
		\]
	\end{defn}
	
	The fact that the doubling constant $D$ can be chosen independently with respect with the underlying set helps us see that for a given uniformly sparse collection $\mathcal{E}$, for each $\epsilon>0$ there is a constant $c$ such that
	\[
	N(E,r)\leq cr^{-\epsilon}
	\]
	for all $r>0, E\in\mathcal{E}.$
	
	\subsection{Some combinatorial results}
	Let $d\geq 1$ be an integer and we consider the direction set $S^{d-1}.$ We are interested in the following problem.
	\begin{ques}\label{Q1}
		Let $A\subset\mathbb{R}^d$ be a compact set such that for each $t\in S^{d-1}\subset\mathbb{R}^d$ there is an affine $(d-1)$-hyperplane $H_t$ normal to $t$ and two points $a,b\in H\cap A$ with $|x-y|>0.001$. What can we say about $\ubox A$?
	\end{ques}
	The number $0.001$ is of no significance, it can be replaced by any fixed positive number. We choose $0.001$ here just for concreteness. We will provide a partial answer to the above problem. Here, we keep $d$ to be a general integer although we will only need the case when $d=3.$ 
	\begin{thm}\label{C1}
		Let $d\geq 3$ be an integer. For $\delta\in (0,1),$ let $S\subset S^{d-1}$ be a $\delta$-separated set with cardinality $N.$ Here we use the spherical metric on $S^{d-1}$. Suppose that for each $t\in S$, there is an affine $(d-1)$-hyperplane $H_t$ with $H_t\perp t$ and two points $a_t,b_t\in H_t\cap [0,1]^d$ with $|a_t-b_t|>0.001.$ Consider $A_\delta=\bigcup_{t\in S} \{a_t,b_t\}.$ Then there is a constant $c$ which does not depend on $\delta$ such that
		\[
		N(A_\delta,\delta)\geq c N^{1/2} \delta^{(d-2)/2}.
		\]
	\end{thm}
	\begin{proof}
		We shall use the pigeonhole principle. Let $\delta<10^{-8}.$ We cover $[0,1]^d$ with almost disjoint $0.01\delta$-cubes (they intersect each other only on boundaries). Let $M>0$ be the largest integer such that there is a $\delta$-cube which contains $M$ elements of the form $a_t,t\in S.$ Fix this cube and we want to consider the corresponding points $b_t.$ There are $M$ of them, but they may not be $\delta$-separated. Let $b$ be one of them. Since the unit vectors normal to the hyperplanes $H_t$ are $\delta$-separated from each other.  In order that $b_t$ occupies the same $0.01\delta$-cubes as $b$ there must be some restrictions on $t.$ Heuristically, consider a line $l\subset\mathbb{R}^d,$ then the family of affine $(d-1)$-hyperplanes containing $l$ forms a $(d-2)$-dimensional family. In $\mathbb{R}^3$, the situation is clear. Similar results hold in higher dimensional Euclidean spaces as well. 
		
		Suppose that we have $0.01\delta$-cubes with separation at least $0.0005.$ These two cubes roughly determine a direction in $S^{d-1}.$ There is a constant $c>0$ such that the direction of the lines which pass both those two cubes are contained in a $c\delta$ ball in $S^{d-1}.$ Since $S$ is $\delta$-separated, for a suitable constant $c$, there are at most $c \delta^{-(d-2)}$ many $t\in S$ such that $H_t$ can intersect both these two cubes. That is to say, there are at least $cM\delta^{d-2}$ many points of form $b_t$ which are $0.01\delta$-separated. On the other hand, as $M$ is the maximum number of points of the form $a_t$ being contained in one $0.01\delta$-cube then there are at least $N/M$ many $0.01\delta$-separated points of the form $a_t.$ Therefore we see that
		\[
		N(\bigcup_{t\in S} \{a_t,b_t\},0.01\delta)\geq \max\{ N/M, c\delta^{d-2}M    \}\geq \sqrt{cN\delta^{d-2}}.
		\]
		This concludes the proof by transforming constants suitably.
	\end{proof}
	
	The above result implies that $\ubox A\geq 1/2$ for Question \ref{Q1}. The argument also leads us to the following result.
	\begin{thm}\label{C2}
		Let $d\geq 3$ be an integer. For $\delta\in (0,1),$ let $S\subset S^{d-1}$ be a $\delta$-separated set with cardinality $N.$ Let $a\in\mathbb{R}^d$. Suppose that for each $t\in S$, there is an affine $(d-1)$-hyperplane $H_t$ with $H_t\perp t$ and a point $b_t\in H_t\cap [0,1]^d$ with $|a-b_t|>0.001.$ Consider $A_\delta=\bigcup_{t\in S} \{b_t\}.$ Then there is a constant $c$ which does not depend on $\delta$ such that
		\[
		N(A_\delta,\delta)\geq c N \delta^{d-2}.
		\]
		In particular, if $N\geq c' \delta^{-(d-1)}$ then $N(A_\delta,\delta)\geq cc' \delta^{-1}.$
	\end{thm}
	
	For proof, notice that the multiplicity $M$ in the proof of the previous theorem can be chosen to be $N$.
	
	We also consider the case when there are some more constraints on $S.$ In this case, we will only prove the following special result. For a smooth space curve $C:t\in [0,1]\to (C_1(t),C_2(t),C_3(t))\in\mathbb{R}^3$ to have nonvanishing torsion if for all $t\in [0,1]$ the following matrix has full rank,
	\[
	Tor_C(t)=
	\begin{pmatrix}
	{C_1}'(t) &{C_2}'(t) & {C_3}'(t)\\
	{C_1}''(t) &{C_2}''(t) & {C_3}''(t)\\
	{C_1}'''(t) &{C_2}'''(t) & {C_3}'''(t)
	\end{pmatrix}.
	\]
	Then the torsion of $C$ at $t$ is
	\[
	tor_C(t)=\frac{\det(Tor_C(t))}{|C'(t)\times C''(t)|^2},
	\]
	whenever the curvature $|C'(t)\times C''(t)|^2\neq 0.$ Here $C'(t)=(C'_1(t),C'_2(t),C'_3(t))$ and $C''(t)=(C''_1(t),C''_2(t),C''_3(t))$. We also used $\times$ to denote the cross product.
	\begin{thm}\label{C3}
		Let $C\subset S^2$ be a smooth curve with nonvanishing curvature and nonvanishing torsion. For $\delta\in (0,1),$ let $D\subset C$ be a $\delta$-separated set with cardinality $N.$ 
		\begin{itemize}
			\item
			Suppose that for each $t\in D$, there is an affine $2$-hyperplane $H_t$ with $H_t\perp t$ and two points $a_t,b_t\in H_t\cap [0,1]^3$ with $|a_t-b_t|>0.001.$ Consider $A_\delta=\bigcup_{t\in S} \{a_t,b_t\}.$ Then there is a constant $c>0$ which does not depend on $\delta$ such that
			\[
			N(A_\delta,\delta)\geq c N^{1/2}.
			\]
			\item
			Let $a\in\mathbb{R}^3$. Suppose that for each $t\in D$, there is an affine $2$-hyperplane $H_t$ with $H_t\perp t$ and a point $b_t\in H_t\cap [0,1]^3$ with $|a-b_t|>0.001.$ Consider $A_\delta=\bigcup_{t\in S} \{b_t\}.$ Then there is a constant $c$ which does not depend on $\delta$ such that
			\[
			N(A_\delta,\delta)\geq cN.
			\]
		\end{itemize}
	\end{thm}
	\begin{rem}
		It is very important that there is a one-dimensional object for us to transfer our counting arguments in the previous proofs. It is likely that one can weaken the smoothness and nonvanishing torsion properties of the curve and replace them with weaker ones.
	\end{rem}
	\begin{proof}
		Up to bounded scaling, we can assume that $D$ is the image under $C$ of a $\delta$-separated subset of $[0,1].$ The argument in the proof of Theorem \ref{C1} can still be performed, but we need to restrict the counting to a curve.
		
		Given $t\in D,$ we see that $a_t,b_t$ defines a direction up to $O(\delta)$-uncertainty. In order that $a_{t'}, b_{t'}$ occupy the same $0.0001\delta$-cube as $a_t,b_t$ we see that $t'=C(s')$ must be contained in a $O(\delta)$-neighbourhood of a great circle on the sphere $S^2.$ Let $R>1$ be a number. Suppose that $C$ is $R$-tangent to this great circle around $t'$ in the sense that $C'(s)$ and the tangent direction of this great circle at $C(s)$ are $O(\delta)$-close for $s$ in a interval of length $R\delta$ centred at $s'.$ Then the direction of $a_t-b_t$ is $\delta$-close to the direction of the binormal vector of $C$ around $s'$, i.e. $B(s)=C'(s)\times C''(s)$ for $s$ in a $R\delta$-interval centred at $s'.$ However, as $|B'(s)|=|tor_C(s)|$ is nowhere vanishing, it has a strictly positive minimum for $s\in [0,1].$ This implies that $R$ can be chosen to be at most $R_C$, a positive value which depends on $C.$ This in turn implies that there are at most $O(1)$ many $t'\in D$ with $a_{t'}, b_{t'}$ being in the same cube as $a_t, b_t.$ Now, we can use the pigeonhole principle. Suppose that there is a $0.0001\delta$-cube with $M$ many points of form $a_t, t\in D.$ Then the corresponding $b_t's$ occupy at least $M/O(1)$ many $0.0001\delta$-cubes. On the other hand, if no $0.0001\delta$-cube contains more than $M$ many points $a_t,t\in D,$ then those points must occupy at least $N/M$ many $0.0001\delta$-cubes. From here,  we see that
		\[
		N(A_\delta,\delta)\geq cN^{1/2}.
		\]
		for a constant $c>0$ which depends on the curve $C$. This shows the first part. The second part follows similarly.
	\end{proof}
	\subsection{Discrepancy theory for irrational rotations}\label{dis}
	Let $(a,b)\in \mathbb{T}^2$ be such that $1,a,b$ are linearly independent over the field of rational numbers. For each $N\geq 1,$ there is a number $D_N(a,b)$ which is minimal with the property that for each ball $B\subset\mathbb{T}^2,$ we have the following estimate
	\[
	\left|\sum_{k=1}^{N} \mathbbm{1}_B(\{ka\},\{kb\})-N\mu(B)\right|\leq N D_N(a,b).
	\]
	It is known that $D_N(a,b)$ decay to $0$ as $N\to\infty.$ The key point here is that $D_N(a,b)$ can be chosen independently with respect to $B$, see \cite[Theorem 1.6]{DT97}. In some cases, it is possible to obtain more explicit upper bounds for $D_N(a,b)$, for example \cite[Theorem 1.80]{DT97}. 
	
	Now let $\epsilon>0$ be a positive number. For a large integer $N$ we consider a set $W\subset \{1,\dots,N\}$ with cardinality at least $\epsilon N.$ Intuitively, we think that
	\[
	\{(\{ka\},\{kb\})\}_{k\in W}
	\]
	forms a large set. To show this, we cover $\mathbb{T}^2$ by disjoint squares with the same side length $r_N.$ We need to choose $r_N$ in a way that the Lebesgue measure of each square is much larger than $D_N(a,b).$ To be concrete, we can choose $r_N=\sqrt{D_N(a,b)}\log D^{-1}_N(a,b).$ We say that $r_N\to 0$ as well. Now each such square $F$ contains between
	\[
	[N\mu(F)-ND_N(a,b),N\mu(F)+ND_N(a,b)]
	\]
	many points in $\{(\{ka\},\{kb\})\}_{k\leq N}.$ All other squares has the same Lebesgue measure $r_N^2$. Suppose that $\{(\{ka\},\{kb\})\}_{k\in W}$ intersect only $K$ many squares. Then $\{(\{ka\},\{kb\})\}_{k\in W}$ has at most 
	\[
	K(Nr^2_N+ND_N(a,b))
	\]
	many points. Therefore we see that,
	\[
	\epsilon N\leq K(Nr^2_N+ND_N(a,b)).
	\]
	This implies that
	\[
	K\geq \frac{\epsilon}{r^2_N+D_N(a,b)}.
	\]
	Now we choose $r_N$ such that $D_N(a,b)\leq 0.001 r^2_N.$ This can be satisfied as \[r^2_N/D_N(a,b)=\log^2 D^{-1}_N(a,b)\to \infty.\] Then we see that
	\[
	K\geq \frac{\epsilon}{1.001} \frac{1}{r^2_N}.
	\]
	Observe that $1/r^2_N$ is roughly the number of disjoint squares we need to cover $\mathbb{T}^2.$
	
	If $1,a,b$ are rationally dependent and $a,b$ are irrational. Then the irrational rotation degenerates to an irrational rotation on a one-dimensional subtorus. In this case, we can use discrepancy estimates on the one dimensional rotation. We omit details and refer the reader to \cite{DT97}.
	
	\subsection{Bernoulli shift}
	Let $\Lambda$ be a finite set of symbols and let $\Omega=\Lambda^{\mathbb{N}}$ be the space of one sided infinite sequences over $\Lambda.$ We define $S$ to be the shift operator, namely, for $\omega=\omega_1\omega_2\dots\in\Omega,$
	\[
	S(\omega)=\omega_2\omega_3\dots.
	\]
	Then we take a $\sigma$-algebra on $\Omega$ generated by cylinder subsets. A cylinder subset $Z\subset\Omega$ is such that
	$
	Z=\prod_{i\in\mathbb{N}}Z_i
	$
	and $Z_i=\Lambda$ for all but finitely many integers $i\in\mathbb{N}.$ We construct a probability measure $\mu$ on $\Omega$ by giving a probability measure $\mu_{\Lambda}=\{p_\lambda\}_{\lambda\in \Lambda}$ on $\Lambda$ and set $\mu=\mu^{\mathbb{N}}_{\Lambda}.$ We require here that $p_\lambda\neq 0$ for all $\lambda\in \Lambda$. Then this system is weak-mixing and has entropy $h(S,\mu)=\sum_{\lambda\in\Lambda} -p_\lambda\log p_\lambda.$ We call this system a Bernoulli shift. We can also introduce a metric topology on $\Omega$ by defining $d(\omega,\omega')=\#\Lambda^{-\min\{i\in\mathbb{N}: \omega_i\neq\omega'_i\}}.$ This turns $\Omega$ into a compact and totally disconnected space. For $\omega\in\Omega$ and $r\in (0,1)$, we use $B(\omega,r)$ to denote the $r$-ball around $\omega$ with radius $r$ with respect to the metric $d$ constructed above.
	
	For more details on Bernoulli shifts and Sinai-Kolmogorov entropy, see, for example, \cite[Section 4]{D11}.
	
	\section{Proof of Theorem \ref{Dyn}}
	Just as in \cite[Section 6]{Y19}, the strategy for proving Theorem \ref{Dyn} contains two main ideas. The first one is to extract a suitable torus rotation out of our Cartesian product of dynamical invariant sets. We will illustrate this idea in full details. The second important idea is an entropy method with Sinai's factor theorem. This method was introduced in \cite{Wu} and modified in \cite{Y19}. After extracting the torus rotation (which is the central part), the application of Sinai's factor theorem will follow the same way as in \cite[Section 10]{Y19}.
	
	Throughout this section, we will assume,  that $a,b,c$ (later on $2,3,5$ for concreteness) are such that $1,\log a/\log b, \log a/\log c$ are $\mathbb{Q}$-linearly independent. This is not a fact, at least by the time of writing. Proceeding in this way helps us to avoid being blocked by technical arguments. After this section, we will remove this strong condition. In fact, we will weaken it further to pairwisely multiplicative independence.
	\subsection{Small set}\label{small}
	Let us first assume that     \[
	s=\Haus A_{a}+\Haus A_{b}+\Haus A_{c}<0.5.
	\]
	
	Let $\mathcal{H}$ be the collection of planes whose normal directions are contained in $S_\Delta.$ We want to show that the collection of subsets $\{H\cap A_a\times A_b\times A_c\}_{H\in\mathcal{H}}$ is uniformly sparse (see Definition \ref{US}). Then the conclusion of this Theorem will follow by applying Proposition \ref{USPA}. Suppose the contrary, there is a positive number $\epsilon$, for an integer $N$ which can be chosen to be arbitrarily large, we can find a plane $H$, a point $\mathbf{x}\in X= H\cap A_a\times A_b\times A_c$ such that $\#W(X,\mathbf{x})\cap [1,N]\geq \epsilon N.$ Let $K\in W(X,\mathbf{x})\cap [1,N].$ We can find a point $\mathbf{y}\in B(\mathbf{x},2^{-K})\setminus B(\mathbf{x},2^{-K-1})$ and $\mathbf{y}\in H$. In what follows we assume that $a=2,b=3,c=5$ concreteness.
	
	Now we want to find a suitable way to re-zoom the whole situation. We can apply $\times 2, \times 3,\times 5$ on the three coordinates respectively. This allow us to extract a nice dynamical system. Let $l,m,n$ be positive numbers such that $1\leq m/l<3, 1\leq n/l<5.$ Define $T(l,m,n)=(l',m',n')$ with
	\[
	l'=2l, m'=\begin{cases}
	m &  \frac{m}{2l}>1\\
	3m & \text{else},
	\end{cases}
	n'=\begin{cases}
	n &  \frac{n}{2l}>1\\
	5n & \text{else}.
	\end{cases}
	\]
	We also define the corresponding linear map $T_{l,m,n}$ by
	\[
	T_{l,m,n}(u,v,w)=(2u,v',w')
	\]
	where $v'=v$ or $v'=3v$ according to whether $m/2l>1$ or not. Similarly, we can define $w'.$ Let $H=\{c_1x+c_2y+c_3z=0\}$ be a plane passing through the origin with $c_1,c_2,c_3\neq 0$. If we apply $T_{l,m,n}$ on $H$, the image is another plane $\{c'_1x+c'_2y+c'_3z=0\}$ with $(c_1,c_2,c_3)=T_{l,m,n}(c'_1,c'_2,c'_3).$ More explicitly, $c'_1=c_1/2$ and $c'_2=c_2$ or $c_2/3$ according to the relation between $l,m.$ Similar result hold for $c'_3$ as well. This also works for planes of form $\{c_1x+c_2y+c_3z+c_4=0\}$. Under the linear map $T_{l,m,n}$ the parameter $c_4$ is also transformed but this will not affect anything. Since $(c_1,c_2,c_3)$ is normal to $H$ we see that $T_{l,m,n}(H)$ has normal direction $T^{-1}_{l,m,n}(c_1,c_2,c_3).$ The action $T$ on $(l,m,n)$ can be viewed as a tori rotation in the logarithmic scale. More precisely, we see that
	\[
	\log (m'/l')=\log (m/l)-\log 2 \mod \log 3\] 
	and
	\[\log (n'/l')=\log (n/l)-\log 2 \mod \log 5.
	\] 
	Since we have assumed that $1, \log 2/\log 3, \log 2/\log 5$ are linearly independent over the field of rational numbers we see that
	\[
	\left(\frac{\log (m'/l')}{\log 3}, \frac{\log (n'/l')}{\log 5}\right)=\left(\frac{\log (m/l)}{\log 3}, \frac{\log (n/l)}{\log 5}\right)-\left(\frac{\log 2}{\log 3},\frac{\log 2}{\log 5}\right) \mod \mathbb{Z}^2.
	\]
	We start with $(l_0,m_0,n_0)=(1,1,1)$ and for simplicity. We have a sequence of linear maps determined by $\{T^{k}(l_0,m_0,n_0)\}_{k\geq 0}.$ We write $T_k=T_{T^{k}(l_0,m_0,n_0)}.$ We also obtain a sequence of planes by putting $H_0=H, H_{k}=T_{k-1}(H_{k-1}).$ Let $t_k=(1,a_k,b_k)$ be normal to $H_k$ write $L (t_k)=(\log a_k,\log b_k).$ Then we see that
	\[
	L (t_{k+1})=L (t_k)+(\log 2/\log 3,\log 2/\log 5) \mod (1,1).
	\]
	We iterate the above procedure $K$ times. As a result, the point $\mathbf{x}$ is sent to $\mathbf{x}'$ and $\mathbf{y}$ is sent to $\mathbf{y}'.$ There are two constants $c_\Delta, C_\Delta$ (which depend only on $\Delta$) such that $c_\Delta<|\mathbf{y}'-\mathbf{x}'|<C_\Delta.$ The plane $H_K$ contains $\mathbf{x}',\mathbf{y}'$ and it is normal to the direction of $2^{-K}(c_1,c_23^{\{K\log 2/\log 3\}},c_35^{\{K\log 2/\log 5\}})$ which is the same as \[(c_1,c_23^{\{K\log 2/\log 3\}},c_35^{\{K\log 2/\log 5\}}).\] Taking $\log$ on each coordinate we obtain the point
	\[
	(\log c_1,\log c_2,\log c_3)+(0,\{K\log 2/\log 3\}\log 3,\{K\log 2/\log 5\}\log 5).
	\] 
	The second term, after a suitable linear transformation looks like $$(\{K\log 2/\log 3\},\{K\log 2/\log 5\}).$$ Since $\log$ is monotone and smooth on $[1,5]$, if there are many choices of integers $K$, then
	\[
	\{(\{K\log 2/\log 3\},\{K\log 2/\log 5\})\}_{K\in W(X,\mathbf{x})\cap [1,N]}
	\]
	forms a rather large set in $\mathbb{T}^2$ and therefore the directions of \[\{(c_1,c_23^{\{K\log 2/\log 3\}},c_35^{\{K\log 2/\log 5\}})\}_{K\in W(X,\mathbf{x})\cap [1,N]}\] forms a large set in $S^{2}.$ More precisely, for each $r>0,$ if $N$ is large enough (in a manner that only depends on $r,\Delta$) and $W(X,\mathbf{x})\cap [1,N]\geq \epsilon N,$ then \[\{(\{K\log 2/\log 3\},\{K\log 2/\log 5\})\}_{K\in W(X,\mathbf{x})\cap [1,N]}\]
	contains at least $0.5\epsilon r^{-2}$ many $r$-separated points, see Section \ref{dis}. Then we see that the directions of \[\{(c_1,c_23^{\{K\log 2/\log 3\}},c_35^{\{K\log 2/\log 5\}})\}_{K\in W(X,\mathbf{x})\cap [1,N]}\] contains a $r$-separated subset of $S^2$ with cardinality $c r^{-2}$ with a suitable constant $c>0.$
	
	Now we have a pair of points $\mathbf{x}',\mathbf{y}'$ which may not be in $[0,1]^3.$ Since $A_2,A_3,A_5$ are $\times 2, \times 3, \times 5 \mod 1$ invariant we can translate $\mathbf{x}',\mathbf{y}'$ together by vectors in $\mathbb{Z}^3.$ As a result, we have found a affine plane $H'_K$, two points $\mathbf{x}'',\mathbf{y}''\in H'_K$ with $|\mathbf{x}''-\mathbf{y}''|\in (c_\Delta,C_\Delta)$ and $\mathbf{x}'',\mathbf{y}'' $ are contained in $A_2\times A_3\times A_5$ or one of its translated copies with translation vector in $\{0,1\}^3.$ By Theorem \ref{C1}, we see that there is a subset of $A_2\times A_3\times A_5$ which is $r$-separated with cardinality
	$
	\gtrsim r^{-0.5}.
	$
	As $r$ can be chosen to be arbitrarily small, this implies that $\ubox A_2+\ubox A_3+\ubox A_5\geq 0.5.$ Hence $\Haus A_2+\Haus A_3+\Haus A_5\geq 0.5,$ a contradiction.
	\subsection{An ergodic sampling result}
	At this stage, one can already prove a weaker version of Theorem \ref{MAIN} with a stronger condition that
	\[
	\frac{\log 2}{\log a}+\frac{\log 2}{\log b}+\frac{\log 2}{\log c}<\frac{1}{2}.
	\]
	Our task now is to replace the requirement `$<1/2$' with `$<1$.' In order to do this, we need an ergodic sampling result which originates from \cite{Wu}. In order to state the result, we introduce the notion of almost Bernoulli property.
	\begin{defn}\label{def1}
		Let $(X,S,\mu)$ be an ergodic system on a compact metric space $X$. Let $\mathcal{A}$ be a finite partition generating the Borel $\sigma$-algebra of $X.$ For $\delta>0,$ we say that $(X,S,\mu)$ is $\delta$-Bernoulli if the following statements hold:
		
		There is a number $c_\delta>0$ and for each integer $n\geq 1,$ there is an integer $N(n)$ and a measurable decomposition  $\mathcal{D}_n=\{D_n(1),\dots,D_n(N(n))\}$ of $X$ such that $(\mathcal{D}^{\mathbb{N}}_n,\pi S,\pi\mu)$ is a Bernoulli shift where $\pi:X\to\mathcal{D}$ is defined by taking $\pi(x)$ to be the sequences of sets $D_n\in\mathcal{D}_n$ such that $S^n(x)\in D_n.$     For each $i\in\{1,\dots,N(n)\},$ we write $\tilde{\mathcal{M}}_n(i)$ to be the collection of atoms of $\mathcal{A}_n$ intersecting $D_n(i).$ For each $i$, there is a subcollection $\mathcal{M}_n(i)\subset \tilde{\mathcal{M}}_n(i)$ consisting at most $c_\delta 2^{n\delta}$ many elements such that the union of atoms in $\cup_{i} \mathcal{M}_n(i)$ has $\mu$ measure at least $1-2\delta.$
		
		We say that $(X,T,\mu)$ satisfies the almost Bernoulli property if it is $\delta$-Bernoulli for each $\delta>0.$
	\end{defn}
	The following result was essentially proved in \cite[Theorem 9.10]{Y19} where $H=\{h_k\}_{k\geq 1}$ was taken to be an irrational rotation orbit on $[0,1]$ but there is no difficulty to show this result for $H$ being equidistributed (with respect to the Lebesgue measure) on $[0,1]^2.$ In the statement, for an integer sequence $K\subset\mathbb{N}$, we write
	\[
	C_K(H)=\overline{\{h_k: k\in K \}}.
	\]
	In what follows, let $X$ be a set and $\mathcal{A}_i,i\geq 1$ be a collection of elements in the power set $\mathcal{P}(X).$ Then $\vee_{i} \mathcal{A}_i$ denotes the smallest $\sigma$-algebra on $X$ that contains all $\mathcal{A}_i,i\geq 1.$
	\begin{thm}\label{thmAustin}
		Let $(X,S,\mu)$ be an ergodic dynamical system with $X$ being a compact metric space. Let $\mathcal{A}$ be a finite partition of $X$ such that $\vee_{i=0}^\infty S^{-i}\mathcal{A}$ generates the Borel $\sigma$-algebra of $X.$ For each $x\in X$ not on the boundaries of sets in $\vee_{i=1}^{n} S^{-i}\mathcal{A}$,  for each $n\in\mathbb{N}$ we denote $A_n(x)$ the unique atom $A$ of $\vee_{i=0}^n S^{-i}\mathcal{A}$ such that $x\in A.$ 
		
		If $\mu$ does not give positive measures to boundaries of $S^{-i}\mathcal{A}$ for all $i\in\mathbb{N}$ and $h(S,\mu)>0, $ then $(X,S,\mu)$ satisfies the almost Bernoulli property. Moreover, let $\epsilon>0$ be arbitrarily chosen in $(0,1)$ and $H=\{h_k\}_{k\geq 1}$ be an equidistributed sequence in $[0,1]^2.$ For each $\delta\in (0,1)$, there is a constant $c_\delta>0$ and $X'_\delta$ with full $\mu$ measure such that  for all $n\geq 1,$ all $x\in X'_\delta$ and all $K\subset\mathbb{N}$ with lower natural density at least $\rho>2\delta+\epsilon,$ there is a collection $\mathcal{M}_n=\mathcal{M}_n(x,K)$ of at most $c_\delta 2^{n\delta}$ atoms of $\mathcal{A}_n$ with the following property: 
		\begin{itemize}
			\item[]
			
			Denote the union of elements in $\mathcal{M}_n$ as $M_n$. We construct the following sequence
			\[
			K'(x)=\{k\in\mathbb{N}: S^k(x)\in M_n \}.
			\]
			Then the following set has Lebesgue measure at least $\epsilon$
			\[
			C_{K\cap K'(x)}(H).
			\]
		\end{itemize}
	\end{thm}
	
	The statement of the above result is very technical. However, the idea behind is intuitive. Let $H$ be a given equidistributed sequence in $[0,1]^2.$ We choose a random subsequence by choosing each term independently with a positive probability. Then it is possible to show that almost surely, the chosen subsequence still equidistributes. Our task now is to replace the random choosing procedure with a deterministic scheme. That is, given a dynamical system $(X,S,\mu),$ a point $x\in X$ and a finite partition of $X$, we follow the trajectory of $x$ under $S.$ As a result, we obtain a sequence of symbols represented by the atoms containing the elements in the trajectory. This sequence of symbols can be treated as an outcome of a coin-tossing procedure. Of course, if $(X,S,\mu)$ is a Bernoulli system with a finite partition of cylinder sets, then we have the same random choosing procedure as discussed before. In general, if $(X,S,\mu)$ be an ergodic system with positive entropy, then by Sinai's factor theorem (see the discussions in \cite{Wu}), one can find a Bernoulli factor with the same entropy. Thus, we can treat $(X,S,\mu)$ essentially as a Bernoulli system with some quantifiable uncertainties. This is where we consider a group of $\lesssim 2^{\delta n}$ atoms as a whole.

	\subsection{Large set}\label{large}
	Now we only require that
	\[
	\Haus A_2+\Haus A_3+\Haus A_5<1.
	\]
	We want to use Theorem \ref{C2} instead of Theorem \ref{C1}. In order to do this, we need to find pairs of points containing in a large range of affine planes with one of the points being trapped in a small set. This can be achieved by using Theorem \ref{thmAustin}. 
	
	Consider the dynamical system with phase space $X=[0,1]^3\times [0,1]^2$ and the map $T:X\to X$ defined by
	\[
	T(x,y,z,u,v)=(x',y',z',u',v')
	\]
	where $x'=\{2x\}$, $u'=\{u+\log 2/\log 3\}$ and $v'=\{v+\log 2/\log 5\}.$ If $u+\log 2/\log 3<1$ then we set $y'=y$ otherwise we set $y'=\{3y\}.$ If $v+\log 2/\log 5<1$ then we set $z'=z$ otherwise we set $z'=\{5z\}.$ Let $H$ be a plane and let $(x,y,z)\in H\cap A_2\times A_3\times A_5.$ Suppose that the normal direction of $H$ is the same as the vector $(1,3^u, 5^v)$, where we have $(u,v)\in [0,1]^2.$ Let $(x',y',z',u',v')=T(x,y,z,u,v).$ Then $(x',y',z')\in H'\cap A_2\times A_3\times A_5$ where $H'$ is a plane whose normal direction is the same as the vector $(1,3^{u'},5^{v'}).$ Let $\mathcal{H}'$ be the collection of planes whose normal direction can be represented by $(1,a,b)$ for $a\in [1,3], b\in [1,5].$ Suppose that $\{H\cap A_2\times A_3\times A_5\}_{H\in\mathcal{H}'}$ is not uniformly sparse. Then there is a positive number $\epsilon>0$ such that for all integers $N$ we can find an integer $N'\geq N$, a plane $H\in\mathcal{H}',$ a point $\mathbf{x}\in X= H\cap A_2\times A_3\times A_5$ such that $\#W(X,\mathbf{x})\cap [1,N']\geq \epsilon N'.$ Suppose that the normal direction of $H$ can be represented by $(1,3^{u},5^{v})$ for $(u,v)\in [0,1]^2$ and let $\mathbf{x}=(x,y,z).$ Then we define the following probability measure on $[0,1]^5$ 
	\[
	\mu_N=\frac{1}{N'} \sum_{i=0}^{N'-1} \delta_{T^{i}(x,y,z,u,v)}.
	\]
	We can take a weak * limit $\mu$ of $\mu_N,N\in\mathbb{N}.$ It can be checked that the component of the last two coordinates of $\mu$ is invariant under the action $$+(\log 2/\log 3, \log 2/\log 5)\mod (1,1).$$ Thus it must be the Lebesgue measure since we have assumed that the above rotation is an irrational rotation. $T$ is not continuous when viewed as a map $\mathbb{T}^3\times \mathbb{T}^2\to \mathbb{T}^3\times \mathbb{T}^2.$ It is discontinuous at points $(x,y,z,u,v)$ where $u=1-\log 2/\log 3$ or $v=1-\log 2/\log 5.$ This is where we are about to choose different multiplication maps for $y,z$ coordinates. However, the projection of $\mu$ to the last two coordinates is the Lebesgue measure we see that $T$ is $\mu$-a.e. continuous. 
	
	Let us consider a function $g: [0,1]^5\to \{0,1\}$ as follows:
	\[
	g(x,y,z,u,v)=1 \iff 1 \in W(H_{u,v}\cap A_2\times A_3\times A_5,(x,y,z)),
	\]
	where $H_{u,v}$ is the plane passing through $(x,y,z)$ normal to $(1,3^u,5^v).$ Since $A_2,A_3,A_5$ are compact, we see that $g$ is measurable and $\{g=1\}$ is closed. By the construction of $\mu$ and the Portmanteau theorem we see that
	\[
	\int g d\mu\geq \epsilon.
	\]
	By taking an ergodic component we can assume $\mu$ to be ergodic.

	After this preparation, we can apply Theorem \ref{thmAustin} just the same way as in \cite[Sections 10.1, 10.2, 10.3]{Y19}. The result we obtain after applying Theorem \ref{thmAustin} is that for each small number $\epsilon'>0,$ for all small enough $r>0,$ there is a subset $A\subset A_2\times A_3\times A_5$ with the following property: \emph{There is a $r$-cube $F$ and a set $A_F\subset A\cap F$ such that for each $\mathbf{x}\in A_F,$ there is a point $\mathbf{y}\in A$ and $\mathbf{x},\mathbf{y}$ are contained in a plane $H_{\mathbf{x},\mathbf{y}}.$ The set of normal directions of planes of form $H_{\mathbf{x},\mathbf{y}}$ forms a $r$-separated subset of $S^{2}$ with cardinality $\geq c r^{-2+\epsilon'}$ where $c$ is a constant which depends on $\Delta,\epsilon'$.} By Theorem \ref{C2}, or directly from the argument in proving Theorem \ref{C1}, we see that
	\[
	N(A,r)\geq c' r^{-1+\epsilon'}
	\]
	where $c'$ is another constant. This implies that $\ubox A_2+\ubox A_3+\ubox A_5\geq 1-\epsilon'.$ Since $\epsilon'$ can be chosen to be arbitrarily small, this contradicts the fact that
	\[
	\ubox A_2+\ubox A_3+\ubox A_5=\Haus A_2+\Haus A_3+\Haus A_5<1.
	\]
	\section{Finishing the proof of Theorem \ref{Dyn}}
	In this section, we will prove Theorem \ref{Dyn} in full generality. So far we have assumed the $\mathbb{Q}$-linear independence for the three numbers $1,\log a/\log b$, $\log a/\log c$. Again, for concreteness we again write $a=2,b=3,c=5.$ Although unlikely, it can happen that $1, \log 2/\log 3, \log 2/\log 5$ are $\mathbb{Q}$-linearly dependent. However, (un)proving this seems to be rather challenging. If this is the case, then the irrational rotation $+(\log 2/\log 3,\log 2/\log 5)\mod \mathbb{Z}^2$ degenerates to an irrational rotation on a subtorus, since $\log 2/\log 3, \log 2/\log 5$ are irrational. Now at the end of the proof of Theorem \ref{Dyn} in Section \ref{small}, the directions of \[\{(c_1,c_23^{\{K\log 2/\log 3\}},c_35^{\{K\log 2/\log 5\}})\}_{K\in W(X,\mathbf{x})\cap [1,N]}\] contains a $r$-separated subset of $S^2$ with cardinality only $cr^{-1}.$ This reflects the fact that the irrational rotation is one-dimensional. The hypothetical rational dependence among $1, \log 2/\log 3, \log 2/\log 5$ now implies that the above set of directions (real projective space), viewed as elements on $X=1$ at pieces of curves of form
	\[
	(1,c_4 3^{t},c_5 5^{qt})
	\]
	for suitable constants $c_4\neq 0,c_5\neq 0$ and a parameter $t$ ranging over a suitable interval in $[0,1]$ and $q$ is a non-zero rational number. In fact, the rotation $+(\log 2/\log 3,\log 2/\log 5)\mod \mathbb{Z}^2$ cannot have horizontal nor vertical line segments as its orbit closure. Otherwise, we would see that either $\log 2/\log 3$ or $\log 2/\log 5$ is rational which are both not the case. From here we see that $q,c_4,c_5\neq 0.$ This is very crucial for later use.
	
	Now we digress to some geometry. Consider the unit sphere $S^2\subset\mathbb{R}^3.$ Let $C:[0,1]\to S^2$ be a smooth curve. Let $x=C(t)\in C$ be a point. We are viewing $S^2$ as the moduli space of $2$-dimensional subspaces. Thus we can consider the affine plane at $X=1$ as a subset of $S^2$ via the natural projection map $y\in\mathbb{R}^3\setminus\{(0,0,0)\}\to y/|y|.$ Without loss of generality, we assume that $C$ is defined entirely on a compact subset of the $X=1$ plane. Suppose that $C$, as a smooth curve on $S^2$, has vanishing torsion at $x$. Intuitively this means that the plane $H(t)$ spanned by $C'(t),C''(t)$ is 'stationary' around $t.$ Now, $H(t)$ is represented as a straight line in the $X=1$ plane.  Thus, vanishing torsion at $x=C(t)$ implies that $C$, as a smooth curve on the $X=1$ plane is tangent to a line with order at least $2.$ That is, there is a line $L$ such that $C$ is tangent to $L$ at $x=C(t)$ and
	\[
	\left|\frac{C'(t+\delta)}{|C'(t+\delta)|}-\frac{C'(t)}{|C'(t)|}\right|=O(\delta^2).
	\]   
	
	Now we come back to the original problem, we see that $C(t)=(c_4 3^t, c_5 5^{qt})$, as a plane curve, has no tangent line at order bigger than $1.$ For plane curves, this reflects the fact that they have non-vanishing curvatures. It is straightforward to check that the non-vanishing condition holds for the curve $C.$ Thus Theorem \ref{C3} can now be employed (for both the cases in Sections \ref{small}, \ref{large})  to deduce the result.
	
	Along the lines of the above argument, we have only used the fact that $2,3,5$ are pairwisely multiplicatively independent. We need the two pairs $2,3$ and $2,5$ to obtain an at-least-one-dimensional irrational rotation. Then we used the pair $3,5$ to conclude that the plane curve $(c_4 3^t,c_5 5^{qt})$ is not tangent to any line of order bigger than $1$, as long as $q$ is a non-zero rational number. Each pair of those multiplicative independence cannot be further dropped.
	
	\section{Further remarks and related problems}
	\subsection{Lower bound estimates}
	We can also ask what happens for Theorem \ref{MAIN} if 
	\[
	s=\frac{\log 2}{\log a}+\frac{\log 2}{\log b}+\frac{\log 2}{\log c}>1.
	\]
	In this case, we can use a dimension decomposition method as in \cite[Section 5.2]{BHY19} to show that Theorem \ref{MAIN} still works in this case but we need to replace $CN^{\epsilon}$ with $CN^{s-1+\epsilon}.$ As we do not actually need this result we omit the proof. What is perhaps more interesting is to see whether the exponent $s-1$ is essentially sharp in this case. To be precise, we pose the following conjecture.
	\begin{conj}
		Let $a,b,c\geq 3$ be pairwise multiplicatively independent integers such that
		\[
		s=\frac{\log 2}{\log a}+\frac{\log 2}{\log b}+\frac{\log 2}{\log c}>1.
		\]
		Then for each $\epsilon>0$ there exists a constant $C$ and for all integer $N$, the number of solutions $(x,y,z)$ with $x,y,z\geq 0$ and $z\leq N$ is at most $CN^{s-1+\epsilon}$ and at least $C^{-1}N^{s-1-\epsilon}.$
	\end{conj}
	\subsection{Numbers with restricted digits in different bases}
	Another interesting (and perhaps more natural) question to think about is numbers with a stronger restriction on digits. For example, let $a,b$ be $2$ multiplicatively independent integers such that
	\[
	\frac{\log 2}{\log a}+\frac{\log 2}{\log b}<1
	\]
	then how large is $B_{a}\cap B_{b}.$ We believe that this intersection is finite, see \cite{BHY19} for more details.
	\begin{conj}
		let $a,b$ be two multiplicatively independent integers such that
		\[
		\frac{\log 2}{\log a}+\frac{\log 2}{\log b}<1.
		\]
		Then $\#B_{a}\cap B_{b}<\infty.$
	\end{conj}
	\section{Acknowledgement}
	HY was financially supported by the University of Cambridge and the Corpus Christi College, Cambridge. HY has received funding from the European Research Council (ERC) under the European Union's Horizon 2020 research and innovation programme (grant agreement No. 803711). We thank the anonymous referee(s) for pointing out a mathematical flaw in an earlier version of this manuscript and other useful comments to make this manuscript much better.

	\providecommand{\bysame}{\leavevmode\hbox to3em{\hrulefill}\thinspace}
	\providecommand{\MR}{\relax\ifhmode\unskip\space\fi MR }
	\providecommand{\MRhref}[2]{%
		\href{http://www.ams.org/mathscinet-getitem?mr=#1}{#2}
	}
	\providecommand{\href}[2]{#2}

\end{document}